\documentclass[12pt, leqno]{article}

\usepackage{amsmath, amssymb, amsthm}

\title{A Generic Polynomial for the Alternating Group $A_5$}
\author{Gene Ward Smith}
\date{}

\newtheorem{theorem}{Theorem}

\begin{document}

\maketitle

\begin{abstract}
The methods of classical invariant theory are used to construct generic
polynomials for groups $S_5$ and $A_5$, along with explicit reductions 
to specializations of the generic polynomials defining any desired field
extension with those groups.
\end{abstract}

\section{The classical invariant theory of binary forms}

The classical invariant theory of binary forms explores the invariants under the
action of ${\mbox{SL}_2}(K)$ in a field $K$ of characteristic 0, but we will 
extend that to consideration of characteristics other than 2, 3, or 5. Initially,
however, it will do no harm to think of $K$ as
$\mathbb{Q}$ or a field of rational functions $\mathbb{Q}$$(t_1, \cdots, t_m)$. 
If a binary form of homogenous degree $n$ in $x$ and $y$ has
a nonzero coefficient for $x^n$, then by setting $y$ to 1, we obtain a polynomial in $x$
of degree $n$. Substituting $x/y$ for $x$ and multiplying by $y^n$ gives the form
back again. Hence the invariant theory of binary forms is also an invariant theory for
polynomials in one variable.

Classical invariant theory actually explores a larger class of invariants, which include 
the {\em covariants}. A covariant of a binary form $f$ is a polynomial of homogenous
degree $r$ (called the order) in $x$ and $y$, with coefficients which are of degree $s$
in the coefficients of $f$, which is invariant in a particular sense under ${\mbox{SL}_2}(K)$. 
The invariants themselves are then the covariants order 0.

If $M = \left(\begin{smallmatrix} a&b\\ c&d \end{smallmatrix} \right)$ is a matrix with determinant 1,
and $g$ a form, then $M(g)$ is the form obtained by substituting $x \mapsto ax+by$, $y \mapsto cx+dy$
for $x$ and $y$. Let $C(f)$ be a function from forms of order $n$ to forms of order $r$, 
whose coefficients are defined
by means of homogenous polynomial functions of degree $s$ in the coefficients of $f$. Then $C(f)$
is a covariant of $f$ if $C$ and $M$ commute, so that $C(M(f)) = M(C(f))$.

Two important special cases of this are when $C(M(f)) = C(f)$ and when $C(f) = f$. In the first case,
$C(f)$ is a form of order 0, and is a true invariant, whose value whether computed from the original or the
transformed coefficients is the same. In the second case, $C(f)$ is the identity map, and
so $C(M(f)) = M(C(f)) = M(f)$; this expresses the fact that $f$ itself is a covariant of order $n$ and
degree 1.

The covariants form a finitely-generated bigraded algebra, whose generators may be found by the operation
of {\rm transvection}. In this era of computer algebra, transvectants are easily computed via Cayley's
$\Omega$-process. This process proceeds by the following steps, starting from two forms $f$ and $g$ in
the variables $x$ and $y$:

\begin{enumerate}

\item
Substitute $x \mapsto x_1, y \mapsto y_1$ in $f$ and $x \mapsto x_2, y \mapsto y_2$ in $g$.

\item
Use the result to define a function $w(x_1, y_1, x_2, y_2) = f(x_1, y_1)g(x_2, y_2)$ in four variables.

\item
Apply the differential operator $\Omega = \frac{\partial^2}{\partial x_1 \partial y_2} -  \frac{\partial^2}{\partial x_2 \partial y_1}$ $m$ times to w. This means that $m$ times in succession, perform the substitution
\[w \mapsto \frac{\partial^2 w}{\partial x_1 \partial y_2} - \frac{\partial^2 w}{\partial x_2 \partial y_1}. \]

\item
Substitute $x_1 \mapsto x$, $x_2 \mapsto x$, $y_1 \mapsto y$, $y_2 \mapsto y$ in the result.

\item
The end result of Cayley's $\Omega$-process is the $m$-th transvectant of $f$ and $g$, written $(f, g)^m$.

\end{enumerate}

\section{The quintic}
The algebra of covariants of the binary quintic has 23 generators, but it is not necessary for us to
consider all of them. If $f = a_0x^5+a_1x^4y+a_2x^3y^2+a_3x^2y^3+a_4xy^4+a_5y^5$ with indeterminate
coefficients, then the Cayley $\Omega$ process staring from $f$ will produce polynomials in $x$ and
$y$ and the indeterminates with integral numerical coefficients. We can divide out the content (the GCD
of all the numerical factors) and obtain a reduced covariant with relatively prime integral coefficients.
We obtain in this way the following covariants:

\begin{itemize}

\item

The original form $f$ of order 5 and degree 1

\item

$i = (f, f)^4/288$ of order 2 and degree 2

\item

$H = (f, f)^2/16$ of order 6 and degree 2

\item

$j = -(f, i)^2/12$ of order 3 and degree 3

\item

$A = (i, i)^2/32$ of order 0 and degree 4

\item

$k = (i, H)^2/12$ of order 4 and degree 4

\item

$\tau = (j, j)^2/16$ of order 2 and degree 6

\item

$B = (\tau, i)^2/8$ of order 0 and degree 8

\item

$\Delta = (A^2-4B)/125$ of order 0 and degree 8

\item

$C = (\tau, \tau)^2/6$ of order 0 and degree 12

\item

$M = (-9C + 2000A\Delta + 1008A^3)/25$ of order 0 and degree 12

\end{itemize}

These covariants, aside from a numerical factor, are the same as the ones with
the same names in \cite[Grace and Young] {GY}. However, the covariant of order and degree 4 is
not there given a name, so here it is called $k$, to go along with $i$ and $j$, of order
and degree both 2 and both 3 respectively. The ring of invariants can be generated
by $A$, $B$, and $C$ and an invariant of degree 18 which we don't require; however
it can equally well be given in terms of $A$, $\Delta$ and $M$, and this turns 
out to be more useful for our purpose. Particular note is drawn to the fact that $\Delta$
is the discriminant.

\section{Generic polynomials}

As we noted in the introduction, the form $f$ in $x$ and $y$ with indeterminate
coefficients $a_0, \cdots, a_5$ corresponds via $y \mapsto 1$ to a polynomial
in $x$. In just the same way, the covariants of $f$ may be converted to corresponding
polynomials. If we do this, $i$ becomes a polynomial of degree 2 in $x$, and $k$
a polynomial of degree 4. We may then apply the Tschirnhausen transformation $z = k/i^2$
and obtain a polynomial of degree 5 in $z$ which gives the same field extension of
$\mathbb{Q}$$(a_0, \cdots, a_5)$. The coefficients of this polynomial are all homogenous of
degree 24 in $a_0, \cdots, a_5$, which means they could be invariants of degree 24.
If that were to be the case, they would be expressible in terms of $A$, $\Delta$ and $M$,
since an invariant of degree 18 can play no role. By direct computation we find that in fact,
the polynomial satisfied by $z$ does have invariants of degree 24 for coefficients, and
is equal to the following polynomial.

\begin{multline}
288M^2z^5+(279890625\Delta^2 A^2+262154475\Delta^4-3666000MA\Delta-\\
2041200MA^3+59541075A^6+87890625\Delta^3+14880M^2)z^4+\\
(3711849300\Delta^4+6170437500\Delta^2 A^2-83428000MA\Delta-23781600MA^3+\\
538658100A^6-351562500\Delta^3+259520M^2)z^3+\\
(15376579650\Delta^4+22131843750\Delta^2 A^2-372984000MA\Delta-\\
130420800MA^3+2685964050A^6+527343750\Delta^3+1583040M^2)z^2+\\
(9952607700\Delta A^4+15161437500\Delta^2 A^2-243612000MA\Delta-\\
79322400MA^3+1619910900A^6-351562500\Delta^3+971040M^2)z-\\
42806000MA\Delta+1743266475\Delta A^4+2912390625\Delta^2 A^2+\\
157216M^2-12805200MA^3+260745075A^6+87890625\Delta^3
\end{multline}

We can rewrite this in terms of a polynomial of degree two in $M$; if we do that we find that the
$M^2$ term is $32 (3z+1)^2 (z+17)^3 M^2$. This suggests replacing $z$ with $(u-1)/3$. 
We might also consider that if we divide the polynomial by $A^6$, we can replace
a polynomial over three indeterminates with one over two: the absolute
invariants (\cite[Elliott]{E}) $\Delta/A^2$ and $M/A^3$, where an
invariant is called ``absolute'' if it is invariant under ${\mbox{GL}_2}(K)$. In place of these,
we will use instead $\delta = 25\Delta/A^2$ and $q = A^3/(8M)$, with an eye to the simplicity of the resulting
polynomial. Performing both of these substitutions, we obtain the following.

\begin{multline}
u^2 (u+50)^3-\\
20 (611  \delta u^3+8505 u^3+39270  \delta u^2+263250 u^2+4050000 u+438000  \delta u+100000  \delta) u q+\\
2 (25+3  \delta) (625  \delta^2 u^4+44550  \delta u^4+793881 u^4+18370800 u^3+3175200  \delta u^3-10000  \delta^2 u^3+\\
262440000 u^2+60000  \delta^2 u^2+25336800  \delta u^2-160000  \delta^2 u+12960000  \delta u+160000  \delta^2) q^2
\end{multline}

By a {\em generic polynomial} outside $S$ with group $G$, where $S$ is a finite set of primes and $G$ a finite group,
is meant a polynomial $P$ in $x$ with coefficients in $\mathbb{Q}$$(t_1, \cdots, t_m)$ such that

\begin{enumerate}

\item

The Galois group of the splitting field for $P$ over $\mathbb{Q}$$(t_1, \cdots, t_m)$ is $G$.

\item

For $p \notin S$, if $P$ is reduced
modulo $p$, the result is defined, separable and irreducible and the
Galois group of the splitting field for $P$ over $\mathbb{F}_p$$(t_1, \cdots, t_m)$ is $G$.

\item

If $K$ is a field of characteristic 0, and $L/K$ is a Galois extension with group $G$,
then $L$ is the splitting field of a polynomial obtained by specializing $(t_1, \cdots, t_m)$ to values
$(\tau_1, \cdots, \tau_m)$ in $K$.

\item

If $K$ is a field of finite characteristic $p \notin S$, and $L/K$ is a Galois extension with group $G$,
then $L$ is the splitting field of a polynomial obtained by reducing $P$ mod $p$ and then
specializing $(t_1, \cdots, t_m)$ to values $(\tau_1, \cdots, \tau_m)$ in $K$.

\end{enumerate}

\begin{theorem}
The two polynomials constructed above are generic for the group $S_5$ outside ${2, 3, 5}$.
\end{theorem}
\begin{proof}
Let us first observe that by \cite[Maeda]{M}, we may treat characteristic $p \neq 2$ uniformly with characteristic 0, and
also that the separability of both polynomials is obvious.

So long as the characteristic is not 2, 3, or 5, the two polynomials above have non-zero discriminant and hence no
repeated roots. 
Since these roots result from a Tschirnhausen transformation, the polynomials have the same Galois group over their
base field as $f$ does over $a_0, \cdots, a_5$, namely $S_5$. In the
other direction, if we start with a polynomial of $g$ degree five with coefficients in $K$ whose roots define $L/K$
with group $G$, then so long as $A$ and $M$ are not zero, the polynomials (1) and (2) above result from a
Tschirnhausen transformation of $g$, and hence give the same extension. 

If either $A$ or $M$ is zero for $g$, we may first apply a preliminary Tschirnhausen transformation and obtain
a polynomial $\tilde{g}$ and obtain $\tilde{A}$ and $\tilde{M}$
which are not zero. An invariant of a polynomial $g$ of order $n$ 
which is of degree $d$ is of weight $w = nd/2$ in the roots of $g$, meaning it is a homogenous function of degree
$w$ in the roots. Hence $A$ may be written as a homogenous polynomial of degree 10 and $M$ of degree 30 in the
roots of $g$. Using resultants, we can eliminate all but one of the roots with respect to $g$, and obtain a polynomial
which must be nonzero for any root of a replacement polynomial $\tilde{g}$ for $g$. These roots which the roots of
$\tilde{g}$ must avoid will lie in $L$,
but there is no reason they should lie in a stem field given by a root of $g$; however even if this should
be so there are only finitely many such roots and hence they are easily avoided. Therefore a suitable $\tilde{g}$
may always be found, in which case its invariants can be used to reduce the extension $L/K$ to one
parametrized by (1) or (2).
\end{proof}

By a theorem of \cite[Kemper]{K}, any polynomial which is generic satisfies the seemingly stronger condition that it
also parameterizes all Galois extensions with groups which are subgroups of $G$. Hence the above polynomials also 
provide the extensions of group $A_5$, just as the polynomial $f$ with indeterminate coefficients $a_0, \cdots, a_5$
will. But they also do something more, which is much more interesting.

\begin{theorem}
If we substitute $\Delta \mapsto D^2$ in the first polynomial, or $\delta \mapsto d^2$
in the second, we obtain a polynomial generic for the group $A_5$ over $\mathbb{Q}$.
\end{theorem}
\begin{proof}
In both cases, the discriminants of the polynomials are squares, as the discriminant of polynomial (1) is a square times $\Delta$,
and of polynomial (2) a square times $\delta$. Since the characteristic is not 2, this means that the Galois group
cannot be $S_5$ and must be contained in $A_5$. Since we can reduce any $A_5$ extension to these forms, the
Galois group must be exactly $A_5$. Moreover in characteristic 0, we may invoke
Maple's ``galois'' function, which can compute Galois groups 
of function field extensions, and compute the Galois group directly; in characteristic $p$, the discriminant 
and the linear resolvent from the sum of two roots suffices. Put succinctly, 
the explicit procedure to reduce any
$A_5$ extension outside of ${2, 3, 5}$ to one of these polynomials plus the fact that the
discriminant is a square shows they are generic.
\end{proof}

\section{Examples}

Suppose our starting polynomial is $x^5-2x^4-10x^3+23x^2-6x-4$, which has Galois group $A_5$. Computing invariants, we find
that $\delta = 25\cdot 72352036/110578^2 = 25/169 = (5/13)^2$, a square. Also, $q = 110578^3/(8\cdot 55159285100995067) = 9343841/3049494563$. 
Substituting these values into polynomial (2), or $d = 5/13$ into the generic $A_5$ polynomial derived from it, we obtain

\begin{multline}
u^5+53018481246319976950/9299417089766560969 u^4+\\
118978291635920447500/9299417089766560969 u^3+\\
131644992415533125000/9299417089766560969 u^2+\\
71941446489050000000/9299417089766560969 u+\\
15555687740000000000/9299417089766560969.
\end{multline}

This we can verify gives the same extension as a root of our original polynomial by factoring
it over the extension field.

Now let us look instead at $x^5+25x^4-x-1$. This has Galois group $S_5$, but $A=0$, and therefore $\delta$ cannot be computed.
However, the polynomial satisfied by the square of the roots, $x^5-625x^4-2x^3+50x^2+x-1$, has $A = 1247920128$, 
$\Delta = -833155976134656$ and $M = 78714822656850046410962239488$, so that $\delta = -2554525/190992984$ and 
$q = 1396546810783488/452524020352953001$. Once again, we may verify we have the same field by factoring over the
extension given by a root.

For a final example, consider $x^5-10c x^3+45c^2 x-c^2$, known as Brioschi normal form. This has absolute invariants
$\delta = 1/5$ and $q = \frac{25}{8192} (1728c-1)/(1764c-1)$. In characteristic 11, we may write the polynomial as
 $x^5+c x^3+c^2 x-c^2$, with absolute invariants $\delta = d^2 = 9$ and $q = 5 (c-1)/(2c+5)$. Since $\delta = d^2$
is a square, we may substitute $d=3$ and $q = 5 (c-1)/(2c+5)$ into the generic $A_5$ polynomial, obtaining

\begin{multline}
u^5+(5 c^2+3 c+5)/(2 c+5)^2 u^4-(3 c^2+3 c+4)/(2 c+5)^2 u^3-\\
(4 c+5)/(2 c+5)^2 u^2-2 (c-5) (c-1)/(2 c+5)^2 u+(3+3 c^2+5 c)/(2 c+5)^2
\end{multline}

We now may verify that both the original and generic parametrized polynomials have
square discriminant, $c^8 (c-1)^2$ and $2^4 (c-1)^4 c^4 (c^3+c^2-5c-1)^2$
respectively, and that the sum of two distinct roots for either satisfies an
irreducible polynomial of degree 10, so that both have $A_5$ as a Galois group. Again,
as before, we can verify their equivalence as extensions of $\mathbb{F}_{11}(c)$
by factoring.

\section{Historical note}
These polynomials and others like them were found by the author in 1987, immediately upon
reading \cite[Noether's Problem for $A_5$]{M} in preprint. It is presented now in the
spirit of better late than never.

\end{document}